\documentclass{article}

\usepackage{t1enc,verbatim,pstool,color,wrapfig}

\usepackage{graphicx}
\usepackage{float}
\usepackage{lipsum}
\usepackage{relsize}
\usepackage{amsmath}
\usepackage{fancyvrb}
\usepackage{latexsym}
\usepackage{booktabs}
\usepackage{amsfonts}
\usepackage[export]{adjustbox}
\usepackage{quoting}
\usepackage{hyperref}
\usepackage[font=small]{caption}
\usepackage{subcaption}
\usepackage{amsthm}
\usepackage{textcomp}
\usepackage{quoting}
\usepackage{amssymb}
\usepackage{epigraph}
\usepackage{geometry}
\usepackage{fancyhdr}
\usepackage{xcolor}
\usepackage{wrapfig}
\newcommand{\rank}{\mathsf{rank}}
\newcommand{\diag}{\mathsf{diag}}

\DeclareMathOperator*{\argmin}{arg\,min }

\newcommand{\scal}[1]{\left \langle #1 \right \rangle}

\def\re{{\mathsf{Re}}}
\def\A{{\mathcal A}}

\def\M{{\mathbb M}}
\def\C{{\mathbb C}}

\def\H{{\mathcal H}}
\def\h{{\mathbb H}}
\def\R{{\mathbb R}}
\def\N{{\mathbb N}}
\def\S{{\mathcal{S}}}
\def\Q{{\mathcal{Q}}}
\def\F{{\mathcal{F}}}
\def\k{\mathbf{k}}

\def\x{{\bf x}}
\def\X{{\bf X}}
\def\y{{\bf y}}

\def\l{{\bf l}}
\def\z{{\bf z}}
\def\b{{\bf b}}
\def\n{{\bf n}}
\def\m{{\bf m}}
\def\j{{\bf j}}
\def\p{{\bf p}}
\def\q{{\bf q}}

\def\rank{{\text{ rank}}}

\def\diag{{\text{diag}}}
\def\prox{{\text{prox}}}

\def\A{{\mathcal{A}}}

\newtheorem{theorem}{Theorem}[section]
\newtheorem{proposition}[theorem]{Proposition}
\newtheorem{lemma}[theorem]{Lemma}

\title{On phase retrieval via matrix completion and the estimation of low rank PSD matrices}
\date{}
\author{Marcus Carlsson\thanks{Lund University, Center for Mathematical Sciences. \newline \texttt{marcus.carlsson@math.lu.se, daniele.gerosa@math.lu.se}}  \and Daniele Gerosa\footnotemark[1]}

\begin{document}
\maketitle

\begin{abstract}
Given underdetermined measurements of a Positive Semi-Definite (PSD) matrix $X$ of known low rank $K$, we present a new algorithm to estimate $X$ based on recent advances in non-convex optimization schemes. We apply this in particular to the phase retrieval problem for Fourier data, which can be formulated as a rank 1 PSD matrix recovery problem. Moreover, we provide theory for how oversampling affects the stability of the lifted inverse problem.
\newline

\newline

\end{abstract}

\section{Introduction}

The (discretized) formulation of the phase retrieval problem consists in finding a complex vector \(\mathbf{x}\), usually a discretized signal or image, given a certain amount of \emph{measurements} in form of modulus of scalar products (see \cite{She}), i.e. \begin{equation}\label{the problem}b_k = | \langle \mathbf{a}_k,\mathbf{x} \rangle |^2, \quad k=1, \dots , M  \end{equation} (where \( \mathbf{a}_k \in \mathbb{C}^N  \) are known), and possibly additional geometric constraints. The aim is thus to reconstruct the discrete vector \( \mathbf{x} \in \mathbb{C}^N  \) representing the object. In \cite{Bal} is shown that \(M \ge 4N -2 \) generic complex measurements are needed in order to be able to distinguish two different signals (up to a phase); in \cite{Con} the lower bound was improved to \( M \ge 4N-4 \). This number is believed to be close to optimal.

A possible combinatorial approach to the problem has been described in \cite{Sah}, together with the proof that, as formulated, the problem is NP-hard. An optimization-based standard solution technique would be to solve the least-squares minimization problem \begin{equation} \min_{\mathbf{x} \in \mathbb{C}^N } \sum_{k=1}^{M} ( | \langle \mathbf{a}_k,\mathbf{x} \rangle |^2 -b_k)^2,    \end{equation} but the quadratic terms in the objective function makes the problem highly non-convex and therefore hard to solve. A convex optimization approach was suggested in \cite{Can}, called PhaseLift, and this note considers an improvement of this approach. We recommend \cite{Fog} for a pleasant overview of recent advances regarding PhaseLift and related techniques.

A lot of work has been devoted to deal with cases in which \( \mathbf{a}_k \) are Gaussian measurements, see e.g. \cite{Can4}, but these algorithms seems to work poorly with Fourier data. The main objective of the present article is to provide improvements of the PhaseLift approach, with a particular focus on Fourier measurements, based on recent advances in non-convex low rank approximation algorithms. The theory relies on the quadratic envelope applied to the indicator functional of matrices with a fixed predetermined rank, and can be used also for any other PSD-estimation problem of fixed low rank (Section \ref{lrpsd}).

\subsection{Physical background}
The phase retrieval problem appears frequently for instance in elecrodynamics; the single scalar complex field \( \Psi (x,y,z,t)  \), often called \emph{wave-function}, solution of the d'Alembert equation \[ \left(  \frac{1}{c^2} \frac{\partial}{\partial t^2 } - \nabla^2   \right) \Psi(x,y,z,t)=0, \]is enough to describe the electromagnetic disturbance in the free space. Making use of Fourier integral, the wave-function can be spectrally decomposed as superposition of monochromatic fields \[ \Psi(x,y,z,t)= \frac{1}{\sqrt{2 \pi}} \int_0^\infty \psi_\omega (x,y,z) \exp(- i \omega t) \, d \omega  \]where the spatial wave-function \( \psi_\omega(x,y,z)  \) associated with a given monochromatic component of the spectral decomposition of \( \Psi \) solves the Helmholz equation \[ ( \nabla^2  +k^2  ) \psi_\omega (x,y,z) = 0, \quad k=\omega/c  \] with suitable boundary conditions. In Coherent Diffractive Imaging (CDI) an object (sample) is illuminated by a coherent almost monochromatic wavefield and subsequentially the diffracted far-field intensity pattern is measured by detectors. The current detectors measure intensities but they are not able to measure the phase of the diffracted pattern; the \emph{phase retrieval problem} consists then in retrieving the object knowing the aforementioned intensity plus, in general, some additional physical constraints.  In the Fraunhofer regime, which occurs whenever the Fresnel number \( N_F := b^2 / \lambda \Delta \ll 1  \) (\(b\) being the diameter of the region occupied by the sample and \(\Delta \) the distance between the sample and the detector) it can be shown (see for instance \cite{Pag}) that the \( 2 \)-dimensional propagated disturbance \( \psi_\omega \) measured at \( z = \Delta \ge 0 \) (where the optical axis is considered to be coincident with the \(z\) axis) is \[ \psi_\omega (x,y,z=\Delta) \longrightarrow - \frac{i k \exp{i k \Delta}}{\Delta} \exp \left[ \frac{i k}{2 \Delta} (x^2 + y^2) \right] \widehat{\psi_\omega} \left( k_x = \frac{k x}{\Delta}, k_y = \frac{ky}{\Delta},z=0 \right)   \]where \(k\) is the wavenumber and \( \widehat{\cdot} \) the \(2\)-dimensional Fourier transform. In words this says that "the propagated disturbance has the form of a paraxial modulated spherical wave with the modulation being proportional to (a transversely scaled form of) the two-dimensional Fourier transform of the unpropagated disturbance" (\cite{Pag}).

A family of algorithms developed for solving the phase retrieval problem in combination with a priori knowledge of the support of the object, relies on the simple idea of alternatingly adjust the support of the image and the modulus of its Fourier transform. This goes back to Gerchberg and Saxton (\cite{Ger}) and to Fienup (\cite{Fie}), even though the pool of mathematical ideas, borrowed from convex analysis, goes back at least to Bregman (\cite{Bre}). The most famous is the \emph{error-reduction algorithm} which is essentially a non-convex adaptation of the idea of Projection On Convex Sets (cfr. POCS algorithm, chatper 5 of \cite{Bau}): if \( \psi_O  \) is the object we want to reconstruct, supported in \( C \subseteq \mathbb{R}^2  \), error-reduction performs successive projections between the two sets \[  \{ \psi \in C^\infty _c \, : \, |\widehat{\psi}|= |\widehat{\psi_O}|  \} \  \text{ and } \ \{ \psi \in C^\infty _c \, : \, \text{supp}(\psi) \subseteq C  \}.   \]The drawback is that the first set is not convex and therefore there are no theoretical guarantees of the convergence of this scheme. Further algorithms have of course been developed; in \cite{Mar} the projections are replaced with reflections, in  \cite{Mare} a Newton-type scheme with Tikhonov regularization is proposed, to mention two recent contributions. These methods have been successfully applied with real data, but the non-convexity issue persists.

\subsection{PhaseLift approach}\label{secphaselift}
 In the present note we focus our attention to a \emph{lifting} approach popularized by \cite{Can} and \cite{Can2}. For \( A_k = \mathbf{a}_k  \mathbf{a}_k ^*  \), where \( {}^*  \) is the conjugate transpose and the latter product is the usual matrix multiplication, we define a linear operator $ \mathcal{A} : \h_N (\mathbb{C}) \to \mathbb{C}^M$  via \begin{equation}\label{lifted equations} X \mapsto \begin{pmatrix} \langle A_1 , X \rangle_F \\ \langle A_2 , X \rangle_F \\ \vdots \\ \langle A_M , X \rangle_F \end{pmatrix}\end{equation} being \( \langle \cdot , \cdot \rangle_F  \) the Frobenius inner product and $\h_N$ the space of $N\times N$ Hermitian matrices. Noticing that \( | \langle \mathbf{a}_k,\mathbf{x} \rangle |^2 = \langle \mathbf{a}_k  \mathbf{a}_k ^* , \mathbf{x}  \mathbf{x} ^* \rangle_F \) and recalling that every positive semidefinite matrix \( X \in \h_N (\mathbb{C})  \) with \(\text{rank}(X)=1\) admits a factorization of the type \(\mathbf{x} \mathbf{x} ^* = X  \) with \( \mathbf{x} \in \mathbb{C}^N \), the quadratic phase retrieval problem is lifted onto a linear one with dimension squared; the problem \eqref{the problem} to be solved is now to find a suitable rank \(1\) PSD-matrix \(X\) such that \[ \mathcal{A} (X) = \mathbf{b}  .\] In real applications, the ``perfect measurement'' $b_k=|\langle \mathbf{a}_k,\mathbf{x} \rangle|^2$ is contaminated by noise, so the problem can be re-casted as an optimization one \begin{equation}\label{minrank} \min \, \|\mathcal{A}(X)-\mathbf{b}\|^2 \ \text{subject to } \text{rank}(X)=1, X \succcurlyeq 0.     \end{equation} The constraint \( \text{rank}(X)=1 \) is however highly non-convex, so a convex relaxation has been proposed in \cite{Can} based on the \emph{nuclear norm}, known for being ``low-rank inducing'' (cfr. \cite{Recht}): \begin{equation}\label{minnuc} \argmin \, \lambda \| X \|_* + \frac{1}{2}\| \mathcal{A}(X) - \mathbf{b} \|_2^2 \ \text{subject to } X  \succcurlyeq 0.  \end{equation} The parameter $\lambda$ gives a tradeoff between low rank of the sought minimum $X$ and data fit, and this term also induces a bias which scales with $\lambda$. For this reason, in practice one does not in general find a rank 1 matrix, but a low rank one from which a suitable vector candidate $\mathbf{x}$ can be extracted, we refer to \cite{Can} for the details. 

\subsection{Innovations and discussion of the lifting trick}

While recognizing the PhaseLift as an important and celebrated contribution, we have found that this method suffers from a number of drawbacks making its practical use limited.
The main issue is the size of the involved objects due to the lifting process. A typical (CDI) measurement generates an image of at least $100\times 100$ which leads to a vector $\mathbf{x}$ of size $10^4$, which in turn yields a $10^4\times 10^4$-matrix which is then fed to an iterative algorithm which need to compute eigenvalue factorization. The time complexity of the latter, without any implementation tricks, is $O((10^4)^3)=O(10^{12})$.

The second drawback stems from the fact that one usually need to run the algorithm several times in order to find a suitable value of $\lambda$, and a third drawback is that one typically need lots of linearly independent measurements $(A_1,\ldots,A_M)$ for stable recovery. As mentioned initially, one needs at least $M=4N$ measurements for having an (essentially) unique solution in the noiseless case. \cite{Can} suggests working with $M\geq cN\log(N)$ where $c$ is some (unknown) constant, whereas in their numerical section they use $M=6N$ where $N=128$. When dealing with Fourier data, one commonly used method for boosting $M$ is simply oversampling in the Fourier domain \cite{Mia}, but it was observed numerically in \cite{Can2} that this gives an ill-posed inverse problem (see Sections 2.1 and 4.4.3), which led them to instead suggest the use of e.g.~masks to increase the amount of measurements.

In this article we will prove that oversampling indeed does not give rise to more \textit{linearly independent} data beyond $M\approx 2^d N$, where $d$ is the dimension of the problem. In other words it is pointless to oversample more than a factor 2 in each dimension. This also shows that any method based on the lifting principle \eqref{lifted equations} is bound to fail unless combined with masks or similar tricks, at least for $d\leq 2$. Furthermore, we will rewrite \eqref{minrank} into an optimization problem where the constrains are built into the functional to be minimized, and then use the \textit{quadratic envelope} to regularize this functional. Ideally one would like to work with the convex envelope of the functional, but this is not doable in practice. The quadratic envelope yields a non-convex continuous surrogate which has recently been shown to have the desirable property of yielding a regularizer which, just like the convex envelope, does not move the global minima \cite{Car2}. While not completely removing other local minimizers, it reduces the amount and in practice it seems that the corresponding algorithm can find the global minimizer in fairly difficult problems. In other words we suggest a new framework which seems to be able to solve the original problem \eqref{minrank}. We refrain from proving this claim mathematically, but remark that in \cite{Car4} an analogous statement is rigorously shown for the similar problem of low rank recovery without PSD constraints. Here we satisfy with developing the algorithm, with a particular focus on efficient implementation, as well as numerically demonstrating several desirable features (see Section \ref{sec numerical});
\begin{itemize}
\item the algorithm can solve any problem of the type \eqref{minrank} but with a predetermined rank $K$, not necessarily $K=1$.
\item The nuclear norm approach \eqref{minnuc} yields matrices with rank larger than 1, hence using more degrees of freedom. Despite this, the proposed method gives a better data fit while at the same time fulfilling the rank constraint.
\item In \cite{Can2} a reweighted version of \eqref{minnuc} is used, which is known for better accuracy, (but this algorithm is non-convex as well). In terms of reconstruction accuracy the two algorithms now perform similarly. Benefits of the proposed algorithm is that it finds the correct rank and does so within fewer iterations, plus that it relies on a more developed theoretical framework.
\end{itemize}

The paper is structured as follows; Section \ref{sec quadratic} presents the mathematical details behind the new general fixed rank PSD-matrix estimation algorithm, Section \ref{sec Fourier} presents our theoretical findings regarding the application to phase retrieval and sampling issues, Section \ref{sec implementation} covers implementation details and finally Section \ref{sec numerical} concludes with some numerical illustrations.

\subsection{Notation}
$\M_N$ will denote the set of $N\times N$ complex matrices and $\h_N$ the subset of Hermitian (i.e.~self-adjoint) matrices. $N$ is the size of the unknown vector \textit{before lifting to a problem with $N^2$ unknowns}, and $M$ is the amount of measurements, which we assume is larger than $N$. Images and higher dimensional objects are represented as tensors $\otimes_{j=1}^d \C^n$ and the measurements then take place in the Fourier domain which we discretize as $\otimes_{j=1}^d \C^m$, where $m\geq n$. Hence $N=n^d$ and $M$ is a constant multiple $N_m+1$ of $m^d$, where $N_m$ is the number of masks. $\A$ denotes the operator \eqref{lifted equations} and its ``matrixification'' (as described in Section \ref{ef grad}) is denoted by $\textbf{A}$, except in the case of pure Fourier data, in which case we denote these objects by $\F$ and $\textbf{F}$ respectively. The quadratic envelope is denoted by $\Q_\gamma(f)$. $\circ$ denotes composition of two functions, in particular if $f$ acts on $\R^N$ then $f\circ\lambda$ acts on $\h_N$, where $\lambda$ denotes the eigenvalues of any given $X\in\h_N$ ordered non-increasingly.

\section{Quadratic envelope approach to low rank recovery}\label{sec quadratic}
	
\begin{wrapfigure}{r}{0.3\textwidth}
\begin{center}
\includegraphics[width=0.28\textwidth]{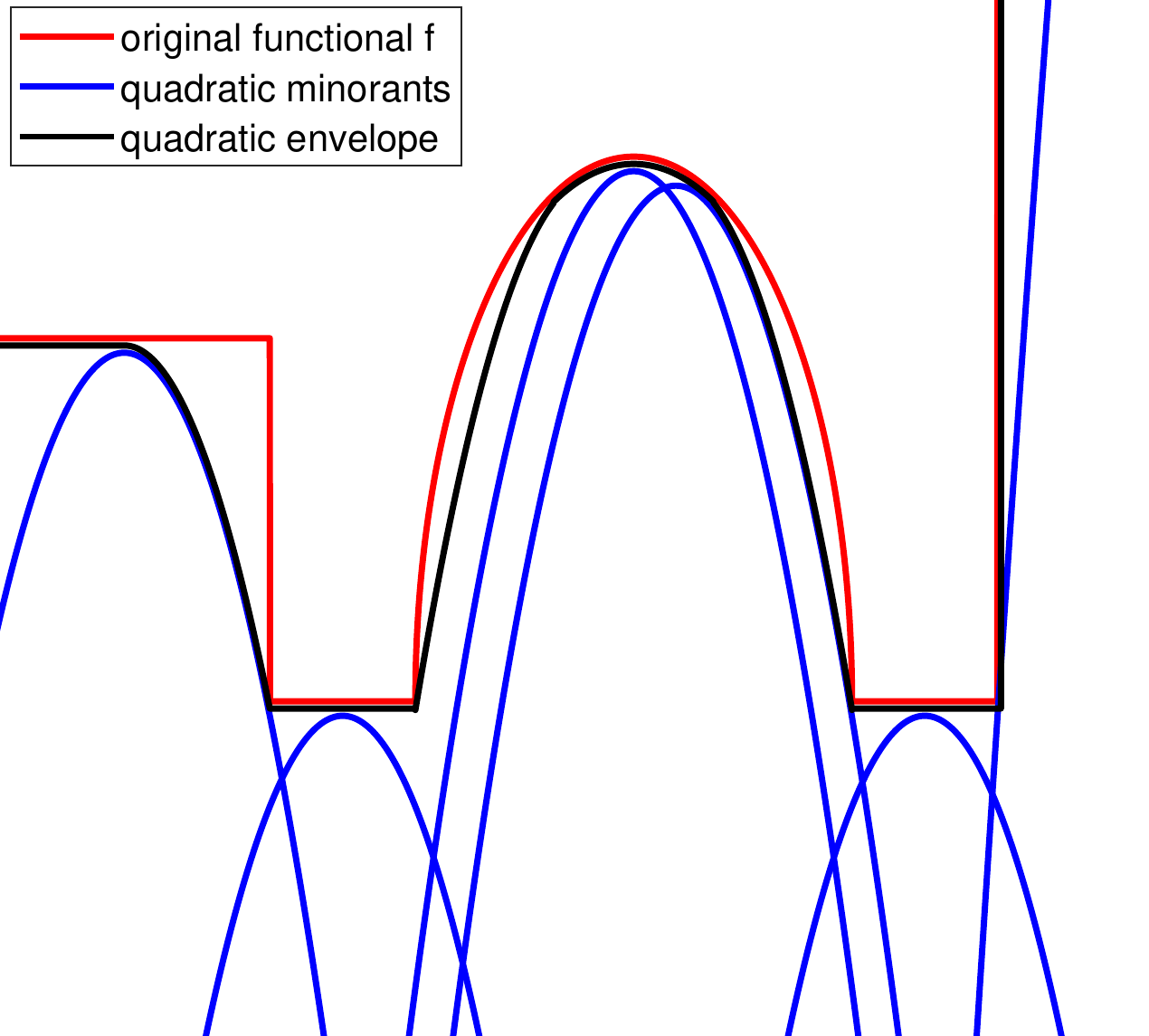}
\end{center}
\caption{Quadratic envelope (black) of a function (red) formed by taking supremum of quadratic minorants (blue)}
\end{wrapfigure}

Let \( \mathcal{H} \) be a Hilbert space, \( f : \mathcal{H} \to \mathbb{R} \cup \{ \infty \} \) a non-convex penalty functional and set \[\mathcal{R}(\x)= f(\x) + \frac{1}{2} \| \mathcal{A}(\x) - \b \|^2. \]  where $\A$ is a linear operator from $\H$ to some other Hilbert space in which the ``measurement'' $\b$ lives. In \cite{Car2} the \emph{quadratic envelope} \[ \mathcal{Q}_\gamma (f)(\x) := \sup_{ \alpha \in \mathbb{R}, \y \in \mathcal{H} } \left\{ \alpha - \frac{\gamma}{2} \| \x- \y\|^2 \, : \, \alpha - \frac{\gamma}{2} \|\cdot - \y\|^2 \le f    \right\} \]	 is studied as a means of regularizing functionals of this type. It is shown that the regularized functional \[ \mathcal{R}_{reg}(\x) = \mathcal{Q}_\gamma (f)(\x) + \frac{1}{2} \| \mathcal{A}( \x) - \b \|^2  \] has some desirable properties when \( \|\mathcal{A}\|^2 < \gamma  \). Most notably, \( \mathcal{R}_{reg}   \) lies between \( \mathcal{R} \) and its lower semi-continuous convex envelope, any local minimizer of \( \mathcal{R}_{reg}  \) is a local minimizer of \( \mathcal{R}  \) and the global minimizers of \( \mathcal{R}  \) and \( \mathcal{R}_{reg}   \) coincide.

In \cite{Car} the quadratic envelope has been studied further for the case of the famous $\ell^0-\ell^2$ problem in which \( \mathcal{H} = \mathbb{R}^N  \) and $f(\x)=\|\x\|_0$ or, in case the sought degree of sparsity $K$ is known,  \[ f(\x) = \iota_K (\x) = \begin{cases} 0 & \text{if } \|\x\|_0 \le K \\ +\infty & \text{otherwise}. \end{cases}  \] Here $\|x\|_0$ denotes the number of non-zero entries in $\x$. In these cases \(\mathcal{R}_{reg} \) has been numerically compared to the more classical \( \ell^1 \)-penalty in solving the problem of retrieving a sparse signal. In the companion work \cite{Car4} the analysis has been lifted to the space of $N\times N$-matrices \( \mathcal{H} = \mathbb{M}_N(\mathbb{C})  \), and \( f(X) = \rank(X) \) or \[ f(X) = \iota_{R_K} (X) = \begin{cases} 0 & \text{if } \rank(X) \le K \\ +\infty & \text{otherwise}. \end{cases}  \] The vector problem and the matrix problem are closely related; note that $\iota_{R_K}(X)=\iota_K(\sigma(X))$ where $\sigma(X)$ denotes the singular values of $X$. It turns out that $$\Q_\gamma(\iota_{R_K})(X)=\Q_\gamma(\iota_K)(\sigma(X))$$ and, more importantly, if $X$ has SVD $X=U\Sigma V^*$, that $\prox_{\iota_{R_K}}(X)=U\diag(\prox_{\iota_K}(\sigma(X))) V^*$. In other words, if we can compute the proximal operator in the scalar case, the matrix case follows immediately. In \cite{Car4} the theory is much further developed for the case $f=\iota_{R_K}$; in particular we show that for noisy measurements the functional $\mathcal{R}_{reg}$ has a unique local minima which is also the solution of $\argmin_{\rank(X)\leq K}\|\A(X)-\b\|$. The new ingredient in the present contribution is basically the PSD condition, and we satisfy with numerically testing the machinery on the Phase Retrieval problem as well as some generic low rank recovery problems with $K>1$.

\subsection{Low-rank Positive Semi-Definite problems}\label{lrpsd}
We now come to the new material for this paper. In this section we set $\H$ to be the set of Hermitian matrices $\h_N \subset \mathbb{M}_N (\mathbb{C})$ and consider the problem of searching for a PSD matrix with fixed rank $K$. Of course, for the PhaseLift problem we will set $K=1$, but we develop the theory for general $K\in\N$.
Motivated by the previous section we introduce the function \( \iota_K^{+} : \mathbb{R}^N \to \{0,+\infty\}  \) defined by \[ \iota_K^{+} (\x) = \begin{cases} 0 & \text{if } \|\x\|_0 \le K \text{ and } \x \ge 0 \\ +\infty & \text{otherwise}. \end{cases}   ,\] where $\x\ge 0$ is interpreted elementwise. If $R_K^+$ denotes the set of positive semi-definite matrices with rank $\leq K$, then it is easy to see that $\iota_{R_K^+}(X)=\iota_{K}^+(\lambda(X))$ where $\lambda(X)$ denotes the eigenvalues of $X\in\H$.
Given a transform $\A$ we then have that the fixed-rank minimization problem
\begin{equation}\label{fix rank}
\argmin_{X\in\h_N^+:~\rank(X)\leq K}\|\A(X)-\b\|^2
\end{equation}
(where $\h_N^+$ denotes the subset of PSD matrices) is equivalent with
\begin{equation}\label{fix rank reformulated}
\argmin_{X\in\h_N}\iota_{R_K^+}(X)+\frac{1}{2}\|\A(X)-\b\|^2.
\end{equation}
This in turn has the same global minimizer as the regularized problem \begin{equation}\label{fix rank regularized}
\argmin_{X\in\h_N}\Q_\gamma(\iota_{R_K^+})(X)+\frac{1}{2}\|\A(X)-\b\|^2
\end{equation}
as long as $\gamma>\|\A\|^2$, as argued earlier. The latter is a continuous (in its domain $\h_N^+$, cfr. Proposition 3.2 of \cite{Car2}) functional whose stationary points can be found for instance via Forward-Backward Splitting scheme (also know as Proximal Gradient Method: see for instance \cite{Bec})\footnote{This statement has a theoretical foundation: for a lower semicontinuous and semialgebric \(f : \mathbb{R}^N \to \mathbb{R} \cup \{ \infty\} \), the function \( \Q_\gamma (f) : \mathbb{R}^N \to \mathbb{R} \cup \{ \infty\} \) is lower semicontinuous and semialgebric (\cite{Car2}), and therefore by a theorem due to Bolte-Daniilidis-Lewis (\cite{Bol}) it has the local Lojasiewicz-Kurdyka property. In \cite{Att} it is showed that the sequence generated by FBS, a particular case of their \textbf{Algorithm 3}, converges to a stationary point of the functional, or else is unbounded. }, as long as the proximal operator of $\Q_\gamma(\iota_{R_K^+})$ is computable. We now describe how this can be done. We recall that a function \( f : \mathbb{R}^N \to \mathbb{R} \) is said to be \emph{symmetric} if \( f(\x)=f(\Pi \x) \) for every permutation \( \Pi \) and for every \( \x \in \mathbb{R}^N \).
We also recall that for a lower semi-continuous function \( g : \mathcal{V} \to [-\infty,\infty]  \), being \( \mathcal{V}  \) an Hilbert space, the (possibly empty or set-valued) \emph{proximal operator} is defined by \[ \text{prox}_g (w) = \argmin_{v \in \mathcal{V}} \left( g(v) + \frac{1}{2} \|v-w\|_{\mathcal{V}} ^2 \right).   \]

\begin{proposition}\label{p1}
Let $f:\R^N\rightarrow [0,\infty]$ be a symmetric function and consider $F:\h_N\rightarrow [0,\infty]$ defined by $F(X)=f(\lambda(X))$. Then $\Q_\gamma(F)=\Q_\gamma(f)\circ \lambda$ and, for $\rho>\gamma$, we have $$\emph{prox}_{\Q_\gamma(F)/\rho}(X)=U \emph{diag}(\emph{prox}_{\Q_\gamma(f)/\rho}(\lambda(X)))U^*,$$
where $U\emph{diag}(\lambda(X))U^*$ is a spectral decomposition of $X$.
\end{proposition}
\begin{proof}
By Proposition 3.1 in \cite{Car2}, we have \( \Q_\gamma (f) = \S_\gamma(\S_\gamma(f)) \), where \( \S_\gamma (f) (y) = \sup_x - f(x) - \frac{\gamma}{2} \| x - y \|^2   \). Note that \[ \S_\gamma (F)(Y)= \sup_X - F(X)  - \frac{\gamma}{2} \|X-Y\|_F ^2 = \sup_X - f(\lambda(X)) - \frac{\gamma}{2} \left( \|X\|_F^2 + \|Y\|_F^2 - 2 \re \langle X,Y \rangle_F \right)  \] and that \( \re \langle X , Y \rangle_F \le \sum_{j=1}^N {\lambda}_j (X) {\lambda}_j (Y)  \), by a variation of the von Neumann trace inequality (cfr.~\cite{Lew}). Moreover the equality holds if and only if there exists a unitary matrix \( V\) such that \( V^* X V = \text{diag}{\lambda}(X) \) and \( V^* Y V = \text{diag}{\lambda}(Y) \). We may therefore assume without loss of generality that \( X=V \text{diag}(\xi) V^* \) for some \( \xi \in \mathbb{R}^N \), being \(V \text{diag}\lambda(Y) V^*  \) a spectral decomposition of \(Y\). For such \(X\) we have that \( \langle X,Y \rangle_F = \sum \xi_j \lambda_j(Y)  \), (where we do not require that \( \xi \) is non-decreasing). By the symmetry of \(f\) we have that \( f(\lambda(X))=f(\xi)  \), and therefore \[ \S_\gamma (F)(Y)= \sup_{\xi \in \mathbb{R}^N  } - f(\xi) - \frac{\gamma}{2} \| \xi - \lambda(Y)\|^2 = \S_\gamma (f) (\lambda(Y)).  \]The corresponding claim for \( \Q_\gamma \) follows immediately by \(  \Q_\gamma (f) = \S_\gamma(\S_\gamma(f)) \).

Let us prove the second part of the statement. We have that \( Z \in \text{prox}_{\Q_\gamma (F)/\rho} (X)  \) if and only if \(Z\) minimizes the functional \( Y \mapsto \Q_\gamma (F) (Y) + \frac{\rho}{2} \|X-Y\|_F ^2   \) which is convex, since \(\Q_\gamma (F) (Y) + \frac{\gamma}{2} \|X-Y\|_F ^2  \)  is the l.s.c. convex envelope of \(F (Y) + \frac{\gamma}{2} \|X-Y\|_F ^2 \) (by Theorem 3.1 \cite{Car2}) and \(\Q_\gamma (F) (Y) + \frac{\rho}{2} \|X-Y\|_F ^2 = \Q_\gamma (F) (Y) + \frac{\gamma}{2} \|X-Y\|_F ^2 + \frac{\rho-\gamma}{2} \|X-Y\|_F ^2\). This in turn happens if and only if \( Z - X \in \partial \Q_\gamma (F) (Z)/\rho = \partial (\Q_\gamma (f )\circ \lambda)(Z) / \rho  \). From Section 5.2 of \cite{Bor} follows that the latter holds if and only if \(Z-X\) and \(Z\) share the same eigenvectors and $$ \lambda(Z)-\lambda(X)=\lambda (Z-X) \in \partial \Q_\gamma (f) (\lambda(Z)) / \rho.$$ This is equivalent to \( \lambda(Z) \in \text{prox}_{\Q_\gamma (f)/\rho} (\lambda(X))  \) and by the symmetry of $f$ we have that $\lambda_i(Z)=\lambda_j(Z)$ whenever $\lambda_i(X)=\lambda_j(X)$. Hence any unitary matrix $U$ of eigenvectors to $X$ are also eigenvectors to $Z$, and since \( Z = U \text{diag}(\lambda(Z)) U^*\), the desired conclusion follows.
\end{proof}

We remark that the maximum negative quadrature of $\Q_\gamma(F)$ is $\gamma$, as shown in \cite{Car2}, so the condition $\gamma<\rho$ ensures that the proximal operator is single valued (since the corresponding minimization problem is strongly convex).

Of course we are interested in the concrete case of $f=\iota_K^+$, but it turns out that $\Q_\gamma(\iota_K^+)$ has a rather complicated expression. Luckily, the related transform $\S_\gamma(\iota_K^+)$ has a simpler expression, and it follows that we can still compute the $\prox_{\Q_\gamma(\iota_K^+)}$ by duality; we postpone the details of this to Section \ref{ef prox}. In the coming section we investigate the structure of the operator $\A$ for the phase retrieval problem with multidimensional data and Fourier measurements.

\section{Images and Fourier data}\label{sec Fourier}

Let us return to the phase retrieval problem. We consider ``images'' in $d$ dimensions, which can be realized as the tensor product vector space $\otimes_{j=1}^d \C^{n}$, where we use the same number of points in every dimension for simplicity only. This linear vector space can of course be identified with $\C^N=\C^{n^d}$ by introducing some basis, but we will see that it is often convenient to actually skip this step and work directly in the more abstract setting. In this case a rank 1 matrix corresponds to a linear operator on $\otimes_{j=1}^d \C^{n}$ of the form $x\otimes y$, and the matrix is PSD if and only if the operator can be written $x\otimes \overline x$, where the bar indicates complex conjugation.

 Given elements $a_{\k}\in \otimes_{j=1}^d \C^{n}$, where $\k$ runs over some set of (usually multidimensional) subindices $S$ (where $M=\# S$), we seek one element $x\in \otimes_{j=1}^d \C^{n}$ such that $$|\scal{x,a_{\k}}|^2=b_{\k},\quad \k\in S.$$
In this new framework, PhaseLift corresponds to the equivalent problem of finding a rank 1 PSD linear operator $X$ on $\otimes_{j=1}^d \C^{n}$ such that \begin{equation}\label{m1}\scal{X,a_{\k}\otimes \overline a_{\k}}=b_{\k},\quad \k\in S.\end{equation}

\subsection{Fourier data and limitations to oversampling}\label{sec oversampling}

In the common case of Fourier measurements, the $a_{\k}$'s are discretizations of pure oscillatory exponential functions, and in this case we denote them by $f_{\k}$ and the corresponding operator by $\F$ in place of $\A$.  We denote by $\ell^2(S)$ the linear vector space of all functions on $S$, so that $b$ naturally identifies with an element of $\ell^2(S)$. Often, we measure on an $m^d$ grid where $(m/n)^d$ is the oversampling factor, i.e. $S=\{0,\ldots, m-1\}^d$ in which case $\ell^2(S)$ is readily identified with $\otimes_{j=1}^d\C^m$. To be more explicit, in this case we have \begin{equation}\label{fkregular}f_\k(\n)=e^{-2\pi i \frac{\k\cdot \n}{m}}\end{equation} for $\k\in \{0,\ldots, m-1\}^d$ and $\n\in\{0,\ldots, n-1\}^d$. However, to allow for unequally spaced sampling we stick to the more general setting where \begin{equation}\label{fk}f_\k(\n)=e^{ i {\zeta_\k\cdot \n}}\end{equation} and $\{\zeta_\k\in\R^d\}_{\k\in S}$ are some frequencies and $S$ some set.
The equation $\eqref{m1}$ can now be written
\begin{equation}\label{m2}\F(X)=b\end{equation}
where $\F: (\otimes_{j=1}^{d} \C^{n}) \otimes (\otimes_{j=1}^{d} \C^{n})\longrightarrow \ell^2(S)$ is the linear operator defined by $\F(X)(\k)=\scal{X,f_{\k}\otimes\overline {f_{\k}}}$.

Since we are working with an ill-posed inverse problem, it is crucial to get as much linearly independent equations as possible. In other words we want to choose $\{\zeta_\k\in\R^n\}_{\k\in S}$ so that $\F$ has maximal rank. The next result basically states that it is pointless to oversample beyond a factor of two.

\begin{lemma}
Independent of how $\{\zeta_\k\in\R^d\}_{\k\in S}$ is chosen, the maximal rank of $\F$ is $(2n-1)^d$.
\end{lemma}
\begin{proof}
We let $\{e_\n\}_{\n\in\{1,\ldots, n\}^{2d}}$ denote the canonical basis in $\otimes_{j=1}^{2d}\C^n$. The range of $\F$ is then spanned by the functions $\F(e_\n)$ which, if we write $\n=(\n^1,\n^2)$ with $\n^1,\n^2\in\{1,\ldots,n\}^d$, takes the form $$\k\mapsto e^{ i {\zeta_\k\cdot (\n^1-\n^2)}}.$$
The amount of tuples of the form $\n^1-\n^2$ equals $(2n-1)^d$, which gives the desired upper bound.
\end{proof}

We now prove that, given sufficient oversampling and a vise choice of $\{\zeta_\k\}_{\k\in S}$, the rank of $\F$ actually equals the maximal rank $(2n-1)^d$. More precisely, we will need that $\{\zeta_\k\}_{\k\in S}$ is such that $\n\mapsto e^{ i {\zeta_\k\cdot \n}}$ is a linearly independent set on $\otimes_{j=1}^d\C^{2n-1}$ or, if $|S|>(2n-1)^d$, that these functions span the full space $\otimes_{j=1}^d\C^{2n-1}$. Such a choice of $\{\zeta_\k\}_{\k\in S}$ will be called non-degenerate. By considering the DFT it is clear that non-degenerate sets of frequencies exist for all cardinalities of $S$. More generally, say that we pick our $\zeta_\k$'s on an irregular product set, i.e.~suppose that for each $j=1\ldots d$ there are ``coordinate-frequencies'' $\{\zeta_{k}^j\}_{k=1}^{2n-1}$ and the multi-frequencies are formed as $\zeta_\k=(\zeta^1_{k_1},\ldots,\zeta^d_{k_d})$ where $\k\in S=\{1,\ldots,2n-1\}^d$. Then, in order for the multidimensional set $\{\zeta_\k\}_{\k\in\{1,\ldots,2n-1\}^d}$ to be non-degenerate it suffices for each coordinate set $\{\zeta_{k}^j\}_{k=1}^{2n-1}$ is non-degenerate (see e.g.~Theorem 2 in \cite{Hay} or Proposition 4.1 in \cite{And}).

\begin{proposition}\label{p2}
Given a non-degenerate set of frequencies $\{\zeta_\k\}_{\k\in S}$, the rank of $\F$ equals \begin{equation}\label{f1}\rank(\F)=\min(|S|,(2n-1)^d).\end{equation}
\end{proposition}
\begin{proof}
First assume that $|S|=(2n-1)^d$ and denote $\F$ in this case by $\F_0$. As in the previous proof we have that the range of $\F_0$ is spanned by $\k\mapsto e^{  i {\zeta_\k\cdot \n}}$ where $\n\in\{-n+1,\ldots,n-1\}$. In this case, there are as many different $\n$'s as there are $\k$'s, and it follows by basic linear algebra that the set of functions of the form $\k\mapsto e^{  i {\zeta_\k\cdot \n}}$ is linearly independent if and only if the set of functions $\n\mapsto e^{  i {\zeta_\k\cdot \n}}$ is, which is true by assumption. This finishes the proof under the assumption that $|S|=(2n-1)^d$.

Given any particular basis for the domain and the codomain, the operator $\F_0$ has a matrix representation $\overrightarrow{\F_0}$ (of dimension $(2n-1)^d\times n^{2d}$). We know from what we have already shown that we have $(2n-1)^d$ linearly independent columns, i.e.~$\overrightarrow{\F_0}$ has full rank. If $|S|>(2n-1)^d$ we can think of $\overrightarrow{\F}$ as adding rows to the matrix $\overrightarrow{\F_0}$, which then impossibly can result in a lower rank. Combined with the previous lemma, this shows that the rank in this case is $(2n-1)^d$. Similarly, if $|S|<(2n-1)^d$ we can think of $\overrightarrow{\F}$ as removing rows from the full rank matrix $\overrightarrow{\F_0}$, which then clearly results in a new full rank matrix. Hence the rank will be $|S|$, and the proof is complete.
\end{proof}

The above proposition is somehow bad news for the "oversampling approach", since it shows that the maximum oversampling factor one could hope for without adding more linearly dependent equations is roughly $2^d$. This conclusion is also backed by the numerical experiments in \cite{Can2}: Section 4.4.3 is devoted to numerically demonstrate that oversampling alone does not yield a well-posed inverse problem. Hence Proposition \ref{p1} can be seen as a theoretical justification of numerical observations in \cite{Can2}. We also validate this conclusion experimentally in Section \ref{sec num over}.
An interesting remark is that two is also a suitable oversampling parameter for other Phase Retrieval methods, see e.g. \cite{Mia} which backs up this conclusion experimentally. This can be understood since $|\hat{x}(\theta)|^2$ is a trigonometric polynomial with $(2n-1)^d$ undetermined coefficients, and it is well known that one needs precisely $(2n-1)^d$
(non-degenerate) measurements to uniquely determine these coefficients. The following interesting papers \cite{Bei,Ben,Bru,Hay,Hua,Osh} contain more information about uniqueness of the Phase Retrieval problem, (as well as other interesting algorithms and plenty of other information).
A curious remark is that the factor $2^d$ does get better with the dimension $d$, and it is also well known in the experimental community that 3d reconstructions are more stable; see e.g.~\cite{Cha,Chu}. For the moment, the case $d=3$ is too computationally expensive with the present method, but we hope to address this shortcoming in future work. The proposition also shows that one does not increase stability, or degree of linear independence to be more precise, by picking $\{\zeta_\k\}$ from some irregularly sampled grid. Hence we find no motivation to deviate from the simplest possible choice, i.e. picking an $m$ in the range $n\leq m <2n$ and use \eqref{fkregular} or, which is the same, setting $S=\{0,\ldots,m-1\}^d$ and $\zeta_\k=\k/m$ in \eqref{fk}.

\subsection{Stabilizing PhaseLift by adding new equations}\label{stabilizing}

Section 2.1 of \cite{Can2} provide a list of experimental methods which can be employed in order to add further linearly independent measurements to those given by the operator $\F$ from the previous section. In particular it is argued that one can use \textit{masks}, that is, you create new measurement vectors $f_\k^C$ by introducing a mask which only lets light through in a region $C$. Mathematically, this amounts to multiplying the image $\x$ with the characteristic function $\chi_C$. If we use regularly sampled measurements as in \eqref{fkregular}, this gives new measurement-functions via the formula $f_\k^{C}(\n)=e^{-2\pi i \frac{\k\cdot \n}{m}}\chi_C(\n)$ where $C$ is realized as a subset of $\{0,\ldots,n-1\}^d$, and $\k$ runs over the index set $\{0,\ldots,m-1\}^d$. Also ptychography, optical grating and oblique illuminations are considered. However additional linearly independent measurements are created, we can form an operator $\A$ by extending $\F$ so that the problems \eqref{fix rank}-\eqref{fix rank regularized} has an essentially unique solution.

While the methods mentioned above are great from a mathematical perspective, they are often not feasible or slow or expensive from a physical perspective in a concrete experiment. Due to this, a priori estimates on the support of the object measured is still by far the most common method used in practice. This was pioneered by Fienup in \cite{Fie} and presently the most popular methods to solve the phase retrieval problem in this context are Hybrid Input Output (\cite{Fie}), Difference Map (\cite{Els}) or Relaxed Averaged Alternating Reflection (\cite{Luk}). A simple way to incorporate support constraints into the scheme \eqref{fix rank}-\eqref{fix rank regularized} is to construct $\A$ by extending $\F$ by simply adding linear equations $X(\m,\n)=0$ for all pairs $(\m,\n)\in (\{0,\ldots,n-1\}^d)^2$ such that either $\m$ or $\n$ is outside of $C$ (recall that $X$ is a tensor in $(\otimes_{j=1}^d\C^{n})\otimes (\otimes_{j=1}^d\C^{n})$). However, this will yield an algorithm which balances meeting the support constraint versus the low rank inducing functional $\Q_\gamma(\iota_{R^+_K})$, and hence the output may still be non-zero off $C$, which may be suitable depending on how certain the support estimate is. An alternative is to set $\A=\F$ (i.e.~only ``pure'' Fourier data) and use ADMM on the problem
\begin{equation}\label{fix rank ADMM}
\argmin_{X\in\h:~X(\m,\n)=0 \text{ for } \m\vee\n\not\in C}\Q_\gamma(\iota_{R_K^+})(X)+\frac{1}{2}\|\F(X)-b\|_{\ell^2(S)}^2
\end{equation}
which will force the solution to obey the support constraint. We leave it for further research to investigate these options from a practical perspective.

\subsection{Estimating $\|\A\|$ for masked Fourier data} \label{parameterschoice}

As explained in section \ref{sec quadratic}, the parameter $\gamma$ in the quadratic envelope needs to be chosen considering the size of $\|\A\|^2$, and since $\A$ in practice is a huge matrix it is not desirable to have to compute its norm. We therefore provide some rough estimates here depending only on the dimension of the images, in the simplest case when $m=n$. We thus assume that $\A$ is formed as explained in the first paragraph of the previous section, using $N_m$ number of masks (plus pure Fourier data). Let $\A_j$, $j=0,\ldots,N_m$ be the suboperator of $\A$ connected to a particular mask (so $\A_0=\F$ are those measurement where no mask is used.) In other words $\A_0(X)(\k)=\langle X,f_\k\otimes \overline{f_\k}\rangle$. The tensors $\{f_\k\otimes \overline{f_\k}:~\k\in\{0,\ldots,n-1\}^d\}$ are mutually orthogonal (since the $f_{\k}$'s are). Moreover $$\|f_\k\otimes \overline{f_\k}\|_2=\|f_\k\|_2^2=n^d=N,$$ so it follows that $\A_0$ is $N$ times a unitary operator, which has operator norm 1. Now, if we represent $\A_0$ as a matrix $\mathbf{A}_0$, then it is clear that $\mathbf{A}_j$ (for each $j\geq 1$), is obtained from $\mathbf{A}_0$ by replacing entire rows and columns by 0. This means that $\|\A_j\|\leq \|\A_0\|=N$ by basic estimates, and hence the triangle inequality implies that $\|\A\|\leq \sum_j\|\A_j\|\leq (N_m+1) N=M$. Summing up we have shown that \begin{equation}\label{estimateA}N\leq \|\A\|\leq M.\end{equation}

\section{Implementation aspects}\label{sec implementation}

In this section we show how to efficiently minimize \eqref{fix rank regularized} without additional constraints, (or at least how to compute a stationary point). Since \eqref{fix rank regularized} is of the type \( f +g  \) where $f=\mathcal{Q}_\gamma (\iota_{R_1 ^+})$ is non-convex but $g=\frac{1}{2}\|\A(\cdot)-b\|^2_{\ell^2(S)}$ is smooth, we use the Forward-Backward Splitting scheme, as this has been established to converge to a stationary point under assumptions which are applicable in our setting. Moreover we suggest to use the accelerated version FISTA, since we have observed that this significantly speeds up convergence. The algorithm alternates between a gradient update step and a proximal operator step. The computation of the gradient of \(g=\frac{1}{2}\|\A(\cdot)-b\|^2_{\ell^2(S)}\) is mathematically straightforward but very time consuming, and we will therefore introduce certain tricks based on FFT and Toeplitz matrices for efficient evaluation of this step when working with Fourier data, see Section \ref{ef grad}. The computation of the proximal operator  of \( \mathcal{Q}_\gamma (\iota_{R_K ^+})(X) = \mathcal{Q}_\gamma (\iota_{K}^+)(\lambda(X))  \)  is rather tricky, the details are given in Section \ref{ef prox}. We give here a general overview of all the functions involved. For concreteness we present the code when working with Fourier data and a number of masks, and leave it up to the reader to work out details for e.g.~ADMM routines with support constraints.

In FISTA, short for Fast Iterative Shrinkage-Thresholding Algorithm, the proximal-gradient steps are taken at the interpolation \[ X^k_{\text{int}}=X^{k+1} + \left( \frac{\theta_k -1}{\theta_{k+1}} \right) (X^{k+1} - X^{k}),   \] where \( \{ \theta_k \}_{ k\ge 1}  \) is a sequence of positive real numbers with some properties that ensure convergence in the convex case (see for instance \cite{Bec}). We used the sequence \( \theta_{k} = (k+1)/2  \), for \( k \ge 1  \) as suggested in Remark 10.35 of \cite{Bec}. Both $X^{k} $ and $X^{k+1}$ are initialized at zero. Upon choosing a step-size $t$ (which we discuss a little further on), the FISTA algorithm now alternates between the update steps \begin{itemize} \item[1.] Compute $X_{\text{int}}^k$.
\item[2.] $X=X^k-t \cdot \text{grad}$, where $\text{grad}$ is the gradient of $\|{\A}(\cdot) - b \|_{\ell^2(S)} ^2$, evaluated at $X^k_{\text{int}}$.
\item[3.] $X_{k+1}=\text{prox}_{\Q_{t\gamma}(\iota_{R_1}^+)}(X)$.
\end{itemize}
For the latter to be single valued, we need to have $\frac{1}{\gamma}>t$, which puts an upper bound to the choice of step-size $t$.

We now discuss suitable choices of $\gamma$ and $t$. Based on \eqref{estimateA} and the theory for the quadratic envelope, it would seem that $\gamma=M^2$ would be a natural choice, since it would guarantee that \eqref{fix rank} and \eqref{fix rank regularized} have the same global minimizer. However, a large $\gamma$ also means that $\Q_\gamma(\iota_{R_K^+})$ is, informally speaking, more non-convex, and we have found that the choice $\gamma=N^2$ gives a better performance. We have also observed that performance is rather stable with respect to changes in $\gamma$ around this value. With this choice, we get the bound $t<1/N^2$ for the step-size. However, the convergence of FISTA is guaranteed (in both convex and non-convex cases) if the stepsize \(t\) is \( < 1 / L(\nabla g ) \) (see \cite{Bau} and \cite{Att}), where \( g\) is the differentiable function and \( L(\nabla g ) \) the Lipschitz constant of its gradient, which in our case leads to $t\leq 1/\|\A\|^2$. Based on \eqref{estimateA}, we therefore used $t=1/(M^2+1)$ in our experiments. In our experience, larger step-size leads to faster convergence, but it is important to not exceed the bound.

%
%
%

\subsection{Efficient computation of the gradient step}\label{ef grad}

We shall here only treat Fourier measurements with masks, in which case the gradient can be computed very efficiently by using the Fast Fourier Transform (FFT in short), but the details are a bit intricate. A technical problem is the representation of tensors as vectors in a computer. Given any enumeration of the index set $\{0,\ldots,n-1\}^d$ and a vector $x\in \otimes_{j=1}^d \C^{n}$, we denote the corresponding vector by $\x\in\C^N$. More precisely, we let $\beta_n:\{0,\ldots,n-1\}^d\rightarrow \{1,\ldots,N\}$ be a bijection and set $\x(\beta_n(\j))=x(\j)$. If for example $d=2$ (so we are dealing with images) and we vectorize by column-stacking, then $\beta_n(\j)=\beta_n(j_1,j_2)=j_2n+j_1+1.$

All vectors are realized as column-vectors, so that e.g.~$\x\x^*$ becomes a matrix. Similarly, if $M$ denotes the amount of elements in $S$ the function $b\in\ell^2(S)$ can be identified with a vector $\b\in\C^M$ by ordering the elements of $S$. Once both the domain $(\C^{n})^d$ and codomain $\ell^2(S)$ have been vectorized, the operators $\A$ and $\F$ get matrix representations that we denote by $\textbf{A}$ and $\textbf{F}$. If $\mathbf{a}_k$ relates to $\A$ (and $\mathbf{A}$) as described shortly before \eqref{lifted equations}, it is now easy to see that \begin{equation} \label{gradient} \nabla  \|\mathbf{A}(\cdot) - \b \|_2 ^2 (\X) = 2\sum_{k=1}^{M} ( \langle \mathbf{a}_k  \mathbf{a}_k^* , \X \rangle_F - b_k)  \mathbf{a}_k \mathbf{a}_k^* , \end{equation}
by basic multivariable calculus. 

We first consider the case of no masks and follow the oversampling recommendations from Section \ref{sec oversampling}; we pick a parameter $m$ with $n\leq m<2n$ and set $S=\{0,\ldots,m-1\}^d$. This set is then in bijective correspondence with $\{1,\ldots,M\}$ through $\beta_m$, and we will implicitly identify $k$ in the latter with $\k=\beta_m(k)$ in the former. Now $\mathbf{A}=\mathbf{F}$ and $\mathbf{a}_k$ derive from pure oscillatory exponentials as in \eqref{fkregular}.  More precisely $\mathbf{a}_k$ corresponds to $a_\k={f}_\k(\j)=e^{-2\pi i \frac{\j\k}{m}}$ as in \eqref{fkregular}, where $a_{\k}$ is the tensor counterpart of $\mathbf{a}_k$.

Similarly, since $\X$ represents an arbitrary matrix in $\M_N$ it implicitly defines a tensor $X$ on the tensor product of $\otimes_{j=1}^d\C^{n}$ with itself, by the formula $X(\j,\l)=\X(\beta_n(\j),\beta_n(\l))$. Summing up we have that $\langle \mathbf{a}_k  \mathbf{a}_k^* , \X \rangle_F$ equals $\langle f_\k \otimes\overline{f_\k} , X \rangle$ in the space $\big(\otimes_{j=1}^d\C^{n}\big)\otimes\big(\otimes_{j=1}^d\C^{n}\big)$ or, more explicitly
 \begin{equation}\label{lod}
\sum_{\j}\sum_{\l} e^{2\pi i\frac{(\l-\j)\k}{m}}\overline{ X(\j,\l)}=\sum_{\p}\sum_{\q} e^{2\pi i\frac{\p\k}{m}}\overline{ X(\q,\p+\q)}=\sum_{\p} e^{2\pi i\frac{\p\k}{m}}\sum_{\q}\overline{ X(\q,\p+\q)}
\end{equation} where the sums over $\j,\l$ run over all ${\{0,\ldots,n-1\}^d}$, $\p$ runs through $\{-n+1,\ldots,n-1\}^d$ and each coordinate of $\q$, say $q_1$, satisfies $\max(0,-p_1)\leq q_1\leq \min(n-1,n-1-p_1)$ to ensure that both $\q$ and $\p+\q$ are in ${\{0,\ldots,n-1\}^d}$. We now introduce $Y(\p)=\sum_{\q}{ X(\q,\p+\q)}$ (where summation bounds for $\q$ are as before) for $\p\in\{-n+1,\ldots,n-1\}^d$, and define $Y(\p)=0$ outside of this grid. In the case $d=1$ the operation above corresponds to summing over lines in $X$ parallel to the diagonal, as in a Toeplitz matrix. For the multivariable case the above is easily computed using $\X$, by simply noting that $Y(\p)=\sum_{\q}{ \X(\beta_n(\q),\beta_n(\p+\q))}$.

Clearly \eqref{lod} is a sort of inverse Fourier transform on $\overline{Y}$, and moreover the computation of $Y$ is fast; $O(N^2)$. To make use of FFT for fast evaluation of this expression jointly for all $\k\in\{0,\ldots,m-1\}^d$, we introduce $Y_{\textsc{mod}}(\p)=\sum_{\n\in\mathbb{Z}^d} f(\p+m\n)$ for all $\p\in\{0,\ldots,m-1\}^d$. Since usually $n\leq m\leq 2m$, note that for each $\p$ the sum contains at most 2 non-zero entries, so this operation is efficiently evaluated. Summing up we have that \begin{equation}\label{fft}\langle \mathbf{a}_k  \mathbf{a}_k^* , \X \rangle_F=\langle \mathbf{f}_\k  \mathbf{f}_\k^* , \X \rangle_F=\overline{DFT({Y_{\textsc{mod}}})(\k)},\end{equation}
where $DFT$ denotes the discrete Fourier transform on the grid $\{0,\ldots,m-1\}^d$. This can be efficiently evaluated (i.e.~$O(N^2\log(N)))$ in most programming languages using some preset implementation of the FFT. For example, when $d=3$ and when using MATLAB, we can represent $Y_{\textsc{MOD}}$ as a multidimensional array and evaluate \eqref{fft} using the command $\textrm{fftn}$.

Let us now return to the formula \eqref{gradient}. We denote the element computed in \eqref{fft} by $c_\k$, and similarly we introduce $b_\k=b_{\beta_m(k)}$. Then \eqref{gradient} takes the form
\begin{equation} \label{gradient1} \nabla  \|\mathbf{A}(\cdot) - \b \|_2 ^2 (\X) = 2\sum_{\k\in\{0,\ldots,m-1\}^{d}} ( c_\k - b_\k)  \mathbf{f}_\k \mathbf{f}_\k^* , \end{equation}
where $\mathbf{f}_\k$ is the vector representation of the tensor $f_\k$ via $\beta_n$, as before. At some fixed index pair $(p,q)\in \{1,\ldots,N\}^2$, this is evaluated as $$2\sum_{\k\in\{0,\ldots,m-1\}^{d}} ( c_\k - b_\k)  \mathbf{f}_\k \mathbf{f}_\k^*=2\sum_{\k\in\{0,\ldots,m-1\}^{d}} ( c_\k - b_\k)e^{2\pi i\frac{(\q-\p)\k}{m}}$$
where $\p=\beta_n(p)$ and $\q=\beta_n(q)$. Clearly all these $N^2=n^{2d}$ values reduce to (at most) $M=m^d$ different values, and can be computed by an inverse Fourier transform on the tensor $(c_\k-b_\k)$, which has time complexity $O(M^2\log(M))$.

This concludes the explanation of how to efficiently compute the gradient step \eqref{gradient} in the case of no masks. Luckily, the incorporation of masks poses very little additional difficulty. Say we have $N_m$ masks $C_1,\ldots, C_{N_m}$ as described in Section \ref{stabilizing}. Clearly then the computation of \eqref{gradient} can be separated in $N_m+1$ independent pieces, one for each mask. We just describe how to compute one of these contributions for a fixed mask $C$. In this case we again have that $k=1,\ldots,m^d$ and we may as well index the vectors using $\k\in\{0,\ldots,m-1\}^d$ instead. Now ${a}_\k$ corresponds to the tensor $\chi_{C}(\j)e^{-2\pi i\frac{\j\k}{m}}$ and $\mathbf{a}_\k$ to its vectorization via $\beta_n$. If $\mathbf{w}$ is the vectorization of $\chi_C$ via $\beta_n$ and $I_\mathbf{w}$ denotes the diagonal operator on $\C^N$ with $\mathbf{w}$ as diagonal, we thus have that $\mathbf{a}_\k=I_{\mathbf{w}}\mathbf{f}_\k$ with $\mathbf{f}_\k$ as before. Returning to the coefficients in \eqref{gradient} we have that $$\langle \mathbf{a}_\k  \mathbf{a}_\k^* , \X \rangle_F=\langle I_{\mathbf{w}} \mathbf{f}_\k  \mathbf{f}_\k^* I_{\mathbf{w}}, \X \rangle_F=\langle  \mathbf{f}_\k  \mathbf{f}_\k^* , I_{\mathbf{w}} \X I_{\mathbf{w}} \rangle_F$$
which means that we can use formula \eqref{fft} upon replacing $\X$ with $I_{\mathbf{w}} \X I_{\mathbf{w}}.$ Finally, since $\mathbf{a}_\k  \mathbf{a}_\k^*= I_{\mathbf{w}} \mathbf{f}_\k  \mathbf{f}_\k^* I_{\mathbf{w}}$ the summation in \eqref{gradient} can be evaluated as in \eqref{gradient1} with the difference that we insert a 0 in the final matrix on positions $(p,q)$ where either of the two coordinates $\p$ and $\q$ lie outside of $C$.

\subsection{Computation of the proximal operator}\label{ef prox}
	The expression of the proximal operator of \( \mathcal{Q}_\gamma (\iota_K ^+)  \) is quite involved. We first recall that \( \mathcal{Q}_\gamma (f) = \mathcal{S}_\gamma ( \mathcal{S}_\gamma(f)) \) where \( \mathcal{S}_\gamma (h)(x) \) is computed by first taking the Legendre transform of \( h + \frac{\gamma}{2}\| \cdot \|^2 \) and then subtracting $\frac{\gamma}{2}\| \cdot \|^2$. By an alteration of the Moreau decomposition (cfr. Chapter 14 of \cite{Bau}) it is possible to show (cfr. \cite{Car3}, Proposition 3.3) that if \( \rho > \gamma \) we have \begin{equation} \label{proxIotaKP} \text{prox}_{\mathcal{Q}_\gamma (\iota_K ^+) / \rho} ( \y ) = \frac{\rho \y - \gamma \cdot \text{prox}_{ \frac{\rho - \gamma}{\rho \gamma} \mathcal{S}_\gamma (\iota_K ^+)} (\y) }{ \rho - \gamma  }  \end{equation} where (see again \cite{Car3}) \[ \S_\gamma (\iota_K ^+)(\x)= \frac{1}{2} \left[ \sum_{j=1}^K |\max(\hat{x}_j,0)|^2 - \|\x\|^2 \right],  \] being \( \hat{\x} \) the decreasing rearrangement of \( \x \).
	We now sketch the routine for computing \( \text{prox}_{ \frac{\rho - \gamma}{\rho \gamma} \mathcal{S}_\gamma (\iota_K ^+)} (\y) \). After some basic algebra it turns out that \[ \text{prox}_{ \frac{\rho - \gamma}{\rho \gamma} \mathcal{S}_\gamma (\iota_K ^+)} (\y) = \argmin_{\x} (\rho - \gamma) \left( \sum_{i=1}^K \max(\hat{x}_i,0)^2 - \| \x \|^2 _2 \right) + \rho \|\x - \y\|^2 _2.  \] By the rearrangement inequality we may assume that \(\y\) is already ordered non increasingly and therefore the latter turns into \begin{equation} \label{prox_int1} \argmin_{x_1 \ge x_2 \ge \dots \ge x_n} (\rho - \gamma) \left( \sum_{i=1}^K \max(x_i,0)^2 - \| \x \|^2 _2 \right) + \rho \|\x - \y\|^2 _2.  \end{equation} For a given \( \x \) let \( \tilde{k}(\x)  \) be the minimum between \(K\) and the last index \(j\) for which \( x_j  \) is non-negative. Then \eqref{prox_int1} becomes \begin{equation} \label{prox_int2} \argmin_{x_1 \ge x_2 \ge \dots \ge x_n} \rho \sum_{i=1}^{ \tilde{k}(\x)  } (x_i^2 - 2 x_i y_i) + \sum_{i=\tilde{k}(\x)+1}^n (\gamma x_i ^2 - 2 \rho x_i y_i).  \end{equation} Now it is clear that if the vector \begin{equation} \label{sol1} \x =  \begin{cases} x_i = y_i & \text{if } i \le \tilde{k}(\y) \\ x_i = \frac{\rho}{\gamma} y_i & \text{if } i > \tilde{k}(\y) \end{cases} \end{equation} is  non-increasing then it is a global minimizer for \eqref{prox_int1}. In particular this happens whenever \(\y\) switches sign before \(K\) (i.e. \(\tilde{k}(\y)<K\)) or whenever \(y_K \ge \rho y_{K+1}/ \gamma \). We exclude these two scenarios from the further analysis.
	
	If \( \tilde{k}(\x) < K  \) then \( x_{\tilde{k}(\x)} \ge 0  \) and \( x_{\tilde{k}(\x)+1}<0  \), but looking at \eqref{prox_int2} we see that setting \( x_{\tilde{k}(\x)+1} = 0  \) would diminish the quantity (since \( y_{\tilde{k}(\x)+1} > 0 \)), a contradiction. Thus we have that \( \tilde{k}(\x)=K  \), and so \eqref{prox_int2} becomes \begin{equation} \label{prox_int3} \underset{x_K \ge 0}{\argmin_{x_1 \ge x_2 \ge \dots \ge x_n}} \rho \sum_{i=1}^{ K  } (x_i^2 - 2 x_i y_i) + \sum_{i=K+1}^n (\gamma x_i ^2 - 2 \rho x_i y_i).  \end{equation} Let \( \x \) be some candidate for the global minimum and set \( s := x_K  \). It is easy to see that \( \x \) must have the following structure: \begin{equation} \label{sol2} \x = \begin{cases} x_i = \max(s,y_i) & \text{if } i \le K \\ x_i= \min(s, \rho y_i / \gamma) & \text{if } i > K; \end{cases}  \end{equation} Now consider \eqref{sol2} inserted into \eqref{prox_int3}; except from an additive constant this function takes the form $$F(s)=\rho \sum_{i=1}^{ K  } (\max(s,y_i)- y_i)^2 + \gamma\sum_{i=K+1}^n (\min(s,\rho y_i/\gamma)  - \rho y_i/\gamma)^2.$$ By inspection it is clear that this function is convex and has a unique minimum on the interval \( V=[y_K,\rho y_{K+1} / \gamma ]\); indeed the first sum is constant on $(-\infty,y_K]$ and strictly convex (increasing) on $(y_K,\infty)$, whereas the second sum is strictly convex (decreasing) $(-\infty,\rho y_{K+1} / \gamma)$ and constant on $[\rho y_{K+1} / \gamma,\infty)$. We seek for the unique solution of \( \frac{d}{d s} F(s)=0  \) in $V$. Let \( j^* \) be the smallest index such that \( y_ {j^*}  \le \rho y_{K+1} / \gamma  \) and let \( l^* \) be the biggest index such that \( \rho y_{l^*} / \gamma \ge y_K  \); now the set \( \{ y_i  \}_{i=j^*}^K \cup \{ \rho y_i / \gamma  \}_{i=K+1}^{l^*}   \) provides a partition of \( V \). Let \(I\) be one of the subintervals. For all values of \( s \) in \(I\) let \(j\) be the first index such that \( x_j = s  \) and let \(l\) be the last, where $x_j$ is given by \eqref{sol2}. The formula for the solution of \( \frac{d}{d s} F(s)=0\) becomes \begin{equation} \label{der_sol} s_I= \frac{ \rho \sum_{i=j}^l y_i}{ (k+1-j)\gamma + (l-k) \rho  }.   \end{equation} If this value is outside \(I\) then the optimal \(s\) is to be found in another interval. By strict convexity \(s_I \in I \) will hold in at least one interval, and at most two intervals (when $s_I$ lies on the boundary).
	\newline
	
	In conclusion we can summarize the previous observations in an algorithm for computing the proximal operator:
	
	\begin{itemize}
		\item[1.] If \(y_K < 0 \) or \( \rho y_{K+1} / \gamma < y_K  \) return \eqref{sol1}, else
		
		\item[2.] compute \(j^* \) and \(l^*\);
		
		\item[3.] sort \(\{ y_i  \}_{i=j^*}^K \cup \{ \rho y_i/\gamma  \}_{i=K+1}^{l^*}\) decreasingly and call it \(\z \);
		
		\item[4.] for \( m\in \{1, \dots , l^* - j^* -1\} \) set \( s = (z_m + z_{m+1})/2 \) and compute the indices \(j\) and \(l\) as described above. Notice that the indices \(j\) and \(l\) are the same for all \( t \in [z_m,z_{m+1}] \), and this is why it is enough to consider the midpoints;
		
		\item[5.] compute the new value of \(s \) according to \eqref{der_sol};
		
		\item[6.] if \( z_m \ge s \ge z_{m+1}  \) stop and return \eqref{sol2}. Otherwise increase \(m\) with \(1\) and repeat the steps 4-5.
	\end{itemize}
	The ``for-loop'' 4-6 must stop since \(F\) is strictly convex and has a unique minimizer.

%
%
	
\section{Numerical examples}\label{sec numerical}
We will illustrate the various results by conducting four experiments. In the first experiment we will compare the proposed method with PhaseLift and reweighted PhaseLift, in the second we test the oversampling conjectures from Section \ref{sec Fourier}, in the third we discuss the capacity of the proposed method of solving low rank PSD problems of any given rank as predicted in Section \ref{sec quadratic}, and finally we end with the reconstruction of a 2D image using the tricks of Section \ref{sec implementation}.

	\subsection{1D synthetic "masked" Fourier measurements}

\begin{figure}[H]
	\begin{minipage}[t]{0.46\textwidth}
		\includegraphics[width=\linewidth]{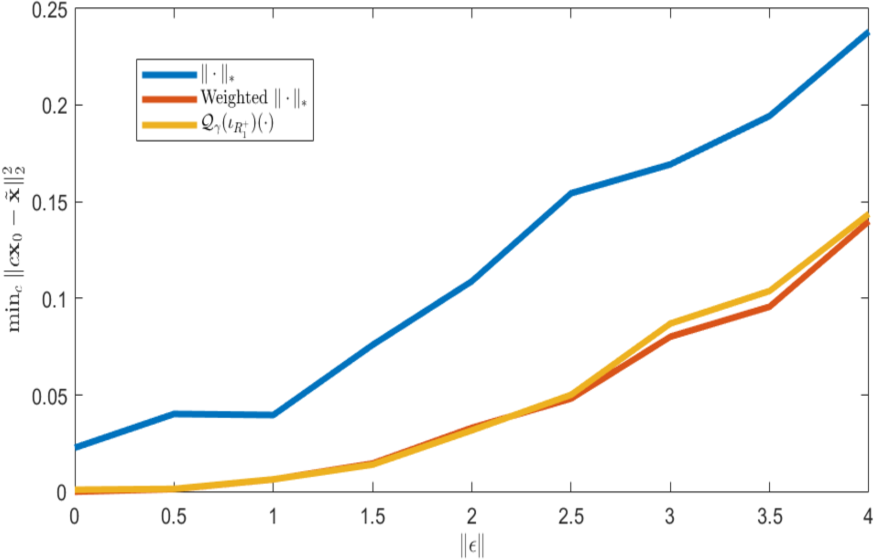}
		\caption{On x-axis the \( \ell^2 \) norm of the noise and on the y-axis \(D\). In blue performances of \eqref{minnuc}, in red of \eqref{reweighted} and in yellow of \eqref{fix rank regularized}}
	\end{minipage}
	\hspace*{\fill} 
	\begin{minipage}[t]{0.48\textwidth}
		\includegraphics[width=\linewidth]{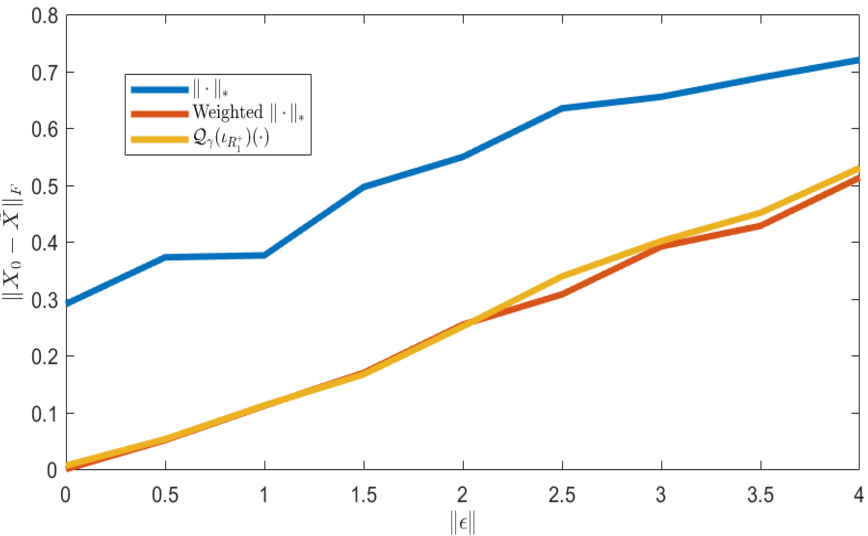}
		\caption{On x-axis the \( \ell^2 \) norm of the noise and on the y-axis the Frobenius norm of \( \tilde{X} - X_0 \). In blue performances of \eqref{minnuc}, in red of \eqref{reweighted} and in yellow of \eqref{fix rank regularized}}\label{fig1}
	\end{minipage}
\end{figure}

	In this subsection we consider 1-dimensional signals \( \x_0 \in \mathbb{C}^{100}  \) with $\|\x_0\|=1$ randomly generated with a Gaussian distribution; they represent our "ground truth", i.e. the signals that  we are interested in retrieving. The method we proposed is compared to the reweighted nuclear norm technique, following the numerical section of \cite{Can2}. The idea of reweighting the nuclear norm was heuristically introduced by \cite{Faz}. Despite of the lack of a systematic mathematical theory behind, this trick seemed to work quite well for our problem too, proving to be able to significantly enhance the capacity of the nuclear norm of finding low rank matrices.
	
	We consider binary masks, as introduced in Section \ref{stabilizing}. Following \cite{Can2} we use $m=n$ (oversampling is considered separately in the next section) and three masks, so $n=m=N=100$ and \(M=400  \). Note that this means that we use 400 equations plus a rank-PSD constraint to determine roughly 5000 variables in the lifted problem.	We focus our attention in comparing the approximated solutions to the two problems \eqref{fix rank regularized} and \begin{equation} \label{reweighted} \min_{X \in \mathbb{H}_{100}} \lambda \| W^{(l)} X \|_* + \frac{1}{2} \| \mathcal{A}(X) - {\b} \|_2 ^2, \ \text{for } l=1,2, \dots  \end{equation} where \(W^{(l)}\) are weight (diagonal) matrices and where $\A$ is as described in section \ref{stabilizing} and \( {\b}=\A(\x_0\x_0^*) + \boldsymbol{\epsilon} \) with \( \epsilon \in \mathbb{C}^{400} \) additive random gaussian noise.

\begin{wrapfigure}{r}{0.4\linewidth}\vspace{-0.4cm}
			\includegraphics[width=\linewidth]{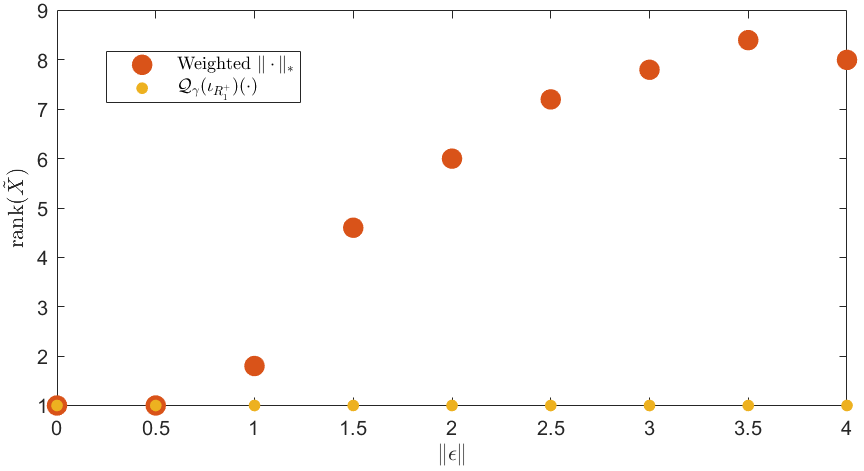}
			\caption{On x-axis the \( \ell^2 \) norm of the noise and on the y-axis the average rank of \( \tilde{X} \). In red performances of \eqref{reweighted}, in yellow of \eqref{fix rank regularized}. The rank was computed using the MATLAB function \texttt{rank} with tolerance \(10^{-6}\) for \eqref{fix rank regularized} and tolerance \(10^{-3} \) for \eqref{reweighted}}\label{fig2}
\end{wrapfigure}

The proposed method only contains one user parameter $\gamma$ which is fixed as described in Section \ref{sec implementation}, independent of noise level. The minimization of \eqref{reweighted} on the other hand relies on $\lambda$ which must scale with the noise for optimal performance, and moreover the update of the weights requires yet another parameter $\delta$ to be chosen; following \cite{Can5} we set \[ w_i^{(l+1)}=\frac{1}{\sigma(X^{(l)})_i + \delta}, \quad \delta >0 \] being \( X^{(l)}  \) the outcome of the algorithm after some fixed number of iterations with $W^{(l)}$. \(W^{(1)} \) is just the identity matrix.

 We have not found any clear guidance in the literature for how to pick $\lambda$ and $\delta$; in \cite{Can2} the algorithm is run several times for different $\lambda$'s picked according to a golden section search or cross validation. This is very time consuming, especially taking into account that the algorithm is slow to begin with due to the lifting trick. We have found that \( \lambda = 0.01 + 0.75  \|\epsilon\| \) and \( \delta = 0.01 + 0.05  \| \epsilon \| \) seemed to work fairly well in our experiments, and these are the values used in the graphs presented. The bisection search only gives a marginal improvement to the results shown, but it would be unfair to do this in comparison with \ref{fix rank regularized}, since this is run only once with one fixed choice of $\gamma$. For practical purposes, the fact that \eqref{fix rank regularized} does not rely on intricate parameter choices is clearly a strength.

	 To compute the numerical approximations we used the FISTA algorithm (already previously described) with \(10000\) iterations for \eqref{fix rank regularized} and with \( 10000 \) iterations for each weights update for \eqref{reweighted}; we found out that \( l>2\) does not give any significant contribution to reconstruction accuracy. The stepsize was fixed to \( t=1/(16N^2+1) = 1 / 1601 \)  (see section \ref{sec implementation}) and 5 different trials for each level of noise were carried out, in the sense that for each \( d \in \{0,0.5,1,1.5,2,2.5,3,3.5,4\}  \) we run five different experiments, i.e. we randomly generated five different ground truths \( X_0\) (and consequently five different noise vectors \(\epsilon\) with \(d=\|\epsilon\|\)), and we averaged the various outcomes over the number of trials. All the masks were randomly generated and the norm of the measured data vector is averagely \( \approx 20 \); in the graphs below we cover cases where the noise has up to \( 20 \% \) of the signal magnitude. According to \cite{Can2} the approximated signal reconstructed using \eqref{reweighted} is chosen as the eigenvector of the greatest eigenvalue of the output \(\tilde{X}\) of the algorithm. Since in \eqref{fix rank regularized} only rank 1 matrices are involved, the approximated signal will be the only nonzero eigenvector of \( \tilde{X} \). The underlying signal is unique only up to a global phase, therefore the distance between \( \x_0  \) and \( \tilde{\x}  \) is computed as \[ D=\min_{c \in \mathbb{C} \cap \{|z|=1\}} \|c \x_0 - \tilde{\x} \|^2 _2.  \]

In Figure \ref{fig1} we plot \(D\), \( \|\tilde{X} - X_0 \|_F \) and in Figure \ref{fig2} the average rank of \( \tilde{X} \). The graphs show that the performances of the three methods. It is clear that nuclear norm without reweighting is inferior in all aspects.
We focus therefore on discussing the proposed method versus reweighted nuclear norm. In terms of ``norms precision'', these are essentially comparable; nevertheless the reweighted nuclear norm fails in finding the correct rank for high levels of noise. Thus, when this method is equivalent with the proposed method, it uses more degrees of freedom. On the other hand the method that we propose is always able to retrieve rank 1 solutions, in perfect compliance with the theory developed.

Since determining the numerical rank is somewhat fishy, we include the following tables which show the first ten eigenvalues (decreasingly ordered) of the approximation \( \tilde{X}  \), computed respectively with \eqref{fix rank regularized} and \eqref{reweighted}, in each of the five different experiments run for the noise level \( \|\epsilon\|=3 \). As is plain to see, the solutions obtained by the reweighted nuclear norm are far from rank 1.

	\begin{table}[H]\parbox{.4\linewidth}{ \centering\begin{tiny}
		\begin{tabular}{||c c c c c c||}
			\hline
			& \textbf{Exp 1} & \textbf{Exp 2} & \textbf{Exp 3} & \textbf{Exp 4} & \textbf{Exp 5}  \\ [0.5ex]
			\hline\hline
			$\lambda_1$ & 0.9402 & 0.9419 & 0.9360 & 0.9461 & 0.9635 \\
			\hline
			$\lambda_2$ & 0.0489 & 0.0579 & -0.0487 & 0.0341 & 0.0439 \\
			\hline
			$\lambda_3$ & -0.0263 & -0.0327 & 0.0339 & 0.0318 & 0.0369 \\
			\hline
			$\lambda_4$ & 0.0260 & -0.0212 & 0.0227 & -0.0298 & 0.0274 \\
			\hline
			$\lambda_5$ & 0.0173 & 0.0158 & -0.0172 & -0.0177 & 0.0155 \\
			\hline
			$\lambda_6$ & -0.0119 & 0.0104 & 0.0100 & 0.0090 & -0.0118 \\
			\hline
			$\lambda_7$ & 0.0051 & 0.0046 & 0.0073 & 0.0061 & 0.0079 \\
			\hline
			$\lambda_8$ & 0 & 0.0003 & 0.0022 & 0.0050 & -0.0021 \\
			\hline
			$\lambda_9$ & 0 & 0 & 0 & 0.0031 & 0 \\
			\hline
			$\lambda_{10}$ & 0 & 0 & 0 & 0 & 0  \\  [1ex]
			\hline
		\end{tabular}\end{tiny}
	\caption{Reweighted nuclear norm eigenvalues table} }\hspace{2cm}
\parbox{.4\linewidth}{\centering \begin{tiny}
	\begin{tabular}{||c c c c c c||}
		\hline
		& \textbf{Exp 1} & \textbf{Exp 2} & \textbf{Exp 3} & \textbf{Exp 4} & \textbf{Exp 5}  \\ [0.5ex]
		\hline\hline
		$\lambda_1$ & 1.0065 & 1.0127 & 1.0340 & 1.0051 & 1.0368 \\
		\hline
		$\lambda_2$ & 0 & 0 & 0 & 0 & 0 \\
		\hline
		$\lambda_3$ & 0 & 0 & 0 & 0 & 0 \\
		\hline
		$\lambda_4$ & 0 & 0 & 0 & 0 & 0 \\
		\hline
		$\lambda_5$ & 0 & 0 & 0 & 0 & 0 \\
		\hline
		$\lambda_6$ & 0 & 0 & 0 & 0 & 0 \\
		\hline
		$\lambda_7$ & 0 & 0 & 0 & 0 & 0 \\
		\hline
		$\lambda_8$ & 0 & 0 & 0 & 0 & 0 \\
		\hline
		$\lambda_9$ & 0 & 0 & 0 & 0 & 0 \\
		\hline
		$\lambda_{10}$ & 0 & 0 & 0 & 0 & 0  \\  [1ex]
		\hline
	\end{tabular}\end{tiny}
\caption{Quadratic envelope eigenvalues table}}
\end{table}

\subsection{Stabilizing by oversampling}\label{sec num over}

In broad terms, the conclusion of Section \ref{sec Fourier} is that it is desirable to oversample a factor 2 (in each dimension), but beyond this point no further gain is expected. We test this for the one dimensional case on synthetic data of size $N=25$ with an identical setup to the one above. 
As measurement we use $\b=\A( \X_0)+\boldsymbol{\epsilon}$ where $\boldsymbol{\epsilon}$ is noise which we vary between 0 and $20\%$ of the magnitude of $\b$. Each data point in the below graphs is the average of five distinct trials.

A numerical observation in \cite{Can2} is that one needs 3 masks to have stable inversion. The authors also point out that a low residual error $\|\A(\tilde \X)-\b\|$ in combination with a poor actual error $\|\tilde \X-\X_0\|$ is an indicator of having an ill-posed inverse problem, i.e.~one where several different points minimize the functional in question.
Below we plot both graphs as a function of the Noise to Signal Ratio $\|\boldsymbol{\epsilon}\|/\|\b\|$.
\begin{figure}[H]
	\begin{minipage}[t]{0.46\textwidth}
		\includegraphics[width=\linewidth]{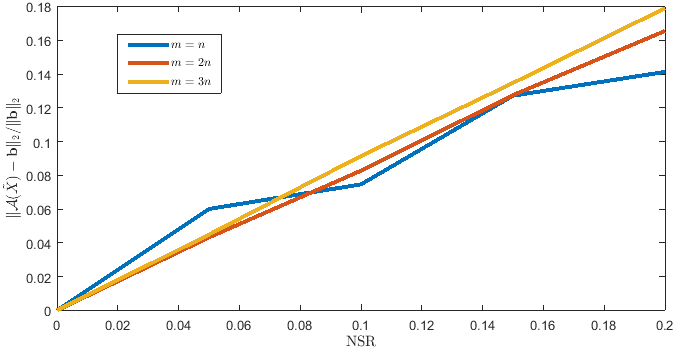}
		\caption{Normalized residual error as a function of normalized noise, 3 masks}
	\end{minipage}
	\hspace*{\fill} 
	\begin{minipage}[t]{0.46\textwidth}
		\includegraphics[width=\linewidth]{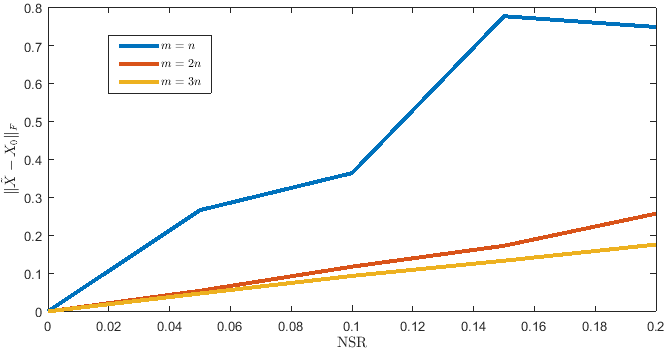}
		\caption{Normalized reconstruction error (distance to ground truth) as function of normalized noise, 3 masks}
	\end{minipage}
\end{figure}

To the left we see that the residual errors are almost identical, and moreover close to the line $y=x$, as expected, since it is clearly unlikely to beat the noise level by any notable margin. Indeed, if we were to get $\tilde \X=\X_0$ then clearly $\|\A(\tilde \X)-\b\|=\|\boldsymbol{\epsilon}\|$, so the above graph actually tells us that we most of the time find a solution better than ground truth.

To the right we see the normalized reconstruction error, and we can confirm that the expectations from Section \ref{sec Fourier} is in perfect compliance with the outcome. Oversampling by a factor two has a significant effect in improving the actual error, whereas oversampling by a factor 3 only has a marginal effect, as predicted by Proposition \ref{p2}. Noteworthy is that also the latter two curves are close to ``$y=x$''.

Encouraged by this result we now try the same setup with 2 masks. As is plain to see, the pattern repeats itself with the major difference that $m=n$ is completely unreliable except for 0 noise, whereas $m=2n$ and $m=3n$ do a similar job which, although not extremely convincing, clearly outperforms $m=n$. In conclusion, it seems that the recommendation of oversampling with a factor of two has both theoretical and numerical support, at least in the case $d=1$.
\begin{figure}[H]
	\begin{minipage}[t]{0.46\textwidth}
		\includegraphics[width=\linewidth]{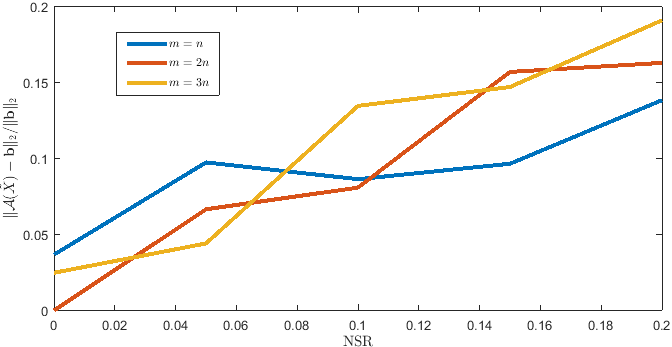}
		\caption{Residual error, 2 masks}
	\end{minipage}
	\hspace*{\fill} 
	\begin{minipage}[t]{0.46\textwidth}
		\includegraphics[width=\linewidth]{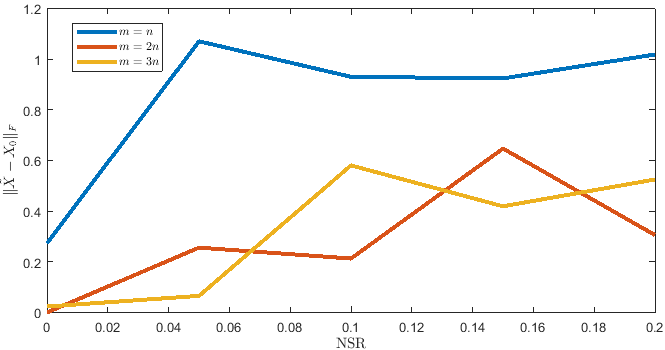}
		\caption{Reconstruction error, 2 masks}
	\end{minipage}
\end{figure}

\subsection{Finding the correct rank for general PSD estimation problems}
We have ran extensive tests for minimizing \eqref{fix rank regularized} with different choices of $\A$, $\b$ and ranks $K$, and never found a single instance where the algorithm converges to a point with the wrong rank, i.e. different from $K$, as long as $\gamma>\|A\|^2$.\footnote{In practice it may be beneficial to violate this restriction, as explained in Section \ref{sec implementation}} This is in accordance with the theory for quadratic envelopes as developed in \cite{Car2} and summarized in Section \ref{sec quadratic}, stating that local minimizers of \eqref{fix rank regularized} form a subset of those for \eqref{fix rank}, except for the fact that FBS is only guaranteed to find \textit{stationary points},\footnote{At least relying on the theory summarized in Section \ref{sec implementation}} not necessarily local minima. This indicates that either there are no saddle point type stationary points, or that FBS actually has the capacity to avoid them. We do not underline these observations with any specific graph, since it seems impossible to design one experiment that would cover all different types of potential applications.

\subsection{2D image reconstruction}
\begin{wrapfigure}{l}{0.12\linewidth}\vspace{-0.4cm}
	\includegraphics[width=\linewidth]{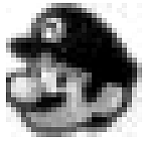}
	\includegraphics[width=\linewidth]{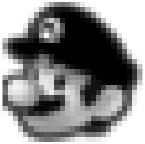}

\end{wrapfigure}
We complement this section by displaying a 2D reconstruction example. The example chosen is a \( 27 \times 27 \) pixels grayscale real image; the initial measurements were contaminated with \( 1 \% \) of additive gaussian random noise. The algorithm used was again FISTA with eight binary masks and the parameters choice indicated in Section \ref{parameterschoice} where we employed the tricks of Section \ref{ef grad} for efficient evaluation via FFT of the gradient-step. The reconstruction was made on a standard laptop and the bottleneck timewise for the algorithm is the eigendecomposition, which prohibits larger images to be processed. We plan to address this shortcoming in future works, which clearly needs to be solved for the algorithm (as well as PhaseLift) to be of practical use for 2D or 3D imaging.

Our experience suggested that the 2D problem needs a higher number of masks (with respect to the 1D case) in order to be stabilized; this observation is present in \cite{Can2} as well where, in the noisefree case, eight binary masks are used. We used $m=n$ for simplicity, it is of course possible that oversampling could lead to fewer masks, but we have not yet tested this.

On the side, the original image is the second while the first is our reconstruction.



\begin{thebibliography}{9}
	\bibitem[AnC17]{And} F. Andersson and M. Carlsson, \emph{On the Structure of Positive Semi-Definite Finite Rank General Domain Hankel and Toeplitz Operators in Several Variables}. Complex Analysis and Operator Theory, 11(4), pp. 755-784, 2017.
	
	\bibitem[ABS13]{Att} H. Attouch, J. Bolte and B. F. Svaiter, \emph{Convergence of descent methods for semi-algebraic and tame problems: proximal algorithms, forward-backward splitting, and regularized Gauss-Seidel methods}. Mathematical Programming, 137(1-2), pp. 91-129, 2013.
	
	\bibitem[BCE06]{Bal} R. Balan, P. Casazza and D. Edidin, \emph{On signal reconstruction without phase}. Applied and Computational Harmonic Analysis, 20(3), pp. 345-356, 2006.
	
	\bibitem[BCo10]{Bau} H. H. Bauschke and P. L. Combettes, \emph{Convex Analysis and Monotone Operator Theory in Hilbert Spaces}. Springer, CMS Books in Mathematics, 2010.
	
	\bibitem[Bec17]{Bec} A. Beck, \emph{First-Order Methods in Optimization}. MOS-SIAM Series in Optimization, 2017.
	
	\bibitem[BeP15]{Bei} R. Beinert and G. Plonka, \emph{Ambiguities in one-dimensional phase retrieval of structured functions}. Proceedings in Applied Mathematics and Mechanics, 15(1), pp. 653-654, 2015.
	
	\bibitem[BBE17]{Ben} T. Bendory, R. Beinert and Y. C. Eldar, \emph{Fourier Phase Retrieval: Uniqueness and Algorithms}. Preprint, \href{https://arxiv.org/abs/1705.09590v3}{arXiv:1705.09590v3}, 2017.
	
	\bibitem[BDL07]{Bol} J. Bolte, A. Daniilidis and A. Lewis, \emph{The Lojasiewicz Inequality for Nonsmooth Subanalytic Functions with Applications to Subgradient Dynamical Systems}. SIAM J. Optim., 17(4), pp.1205-1223, 2007.
	
	\bibitem[BoL05]{Bor} J. M. Borwein and A. S. Lewis, \emph{Convex Analysis and Nonlinear Optimization - Theory and Examples}. CMS books in Mathematics, 2005.
	
	
	\bibitem[Bre65]{Bre} L. M. Bregman, \emph{The method of successive projections for finding a common point of convex sets}. Soviet Math. Dokl., 6, pp. 688-692, 1965.
	
	\bibitem[BrS79]{Bru} Yu. M. Bruck and L. G. Sodin, \emph{On the ambiguity of the image reconstruction problem}. Optics Communications, 30(3), pp. 304-308, 1979.
	
	\bibitem[CES13]{Can2} E. J. Cand\`es, Y. C. Eldar, T. Strohmer and V. Voroninski, \emph{Phase Retrieval via Matrix Completion}. SIAM J. Imaging Sci., 6(1), pp. 199-225, 2013.
	
	\bibitem[CSV13]{Can} E. J. Cand\`es, T. Strohmer and V. Voroninski, \emph{Phaselift: Exact and stable signal recovery from magnitude measurements via convex programming}. Commun. Pure Appl. Math., vol. 66, no. 8, pp. 1241-1274, 2013.
	
	\bibitem[CaL14]{Can3} E. J. Cand\`es and X. Li, \emph{Solving quadratic equations via PhaseLift when there are about as many equations as unknowns}. Foundations of Computational Mathematics, 14(5), pp. 1017-1026, 2014.
	
	\bibitem[CLS15]{Can4} E. J. Cand\`es, X. Li and M. Soltanolkotabi, \emph{Phase Retrieval via Wirtinger Flow: Theory and Algorithms}. IEEE Transactions on Information Theory, 61(4), pp. 1985-2007, 2005.
	
	\bibitem[CWB08]{Can5} E. J. Cand\`es, M. B. Wakin and S. P. Boyd, \emph{Enhancing Sparsity by Reweighted $\ell^1$ Minimization}. Journal of Fourier Analysis and Applications, 14(5-6), pp. 877–905, 2008.

	\bibitem[CGO18]{Car} M. Carlsson, D. Gerosa and C. Olsson, \emph{An unbiased approach to compressed sensing}. Preprint, \href{https://arxiv.org/abs/1806.05283}{arXiv:1806.05283v4}, 2019.
	
	\bibitem[CGO19]{Car4} M. Carlsson, D. Gerosa and C. Olsson, \emph{An unbiased approach to low rank recovery}. Arxiv preprint, identifier 2865217.
	
	\bibitem[Car19]{Car2} M. Carlsson, \emph{On Convex Envelopes and Regularization of Non-convex Functionals Without Moving Global Minima.} J Optim Theory Appl, 183(1), pp. 66-84, 2019.
	
	\bibitem[Car16]{Car3} M. Carlsson, \emph{On convexification/optimization of functionals including an \( \ell^2 \)-misfit term.} Preprint, \href{https://arxiv.org/pdf/1609.09378.pdf}{arXiv:1609.09378v4}.
	
\bibitem[CBM06]{Cha} H. Chapman, A. Barty, S. Marchesini, A. Noy, S. Hau-Riege, C. Cui, M. Howells, R. Rosen, H. He, J. Spence, U. Weierstall, T. Beetz, C. Jacobsen and D. Shapiro.
 \emph{High-resolution ab initio three-dimensional x-ray diffraction microscopy}. J. Opt. Soc. Am. A 23.5, pp. 1179--1200, 2006.

\bibitem[CZL14]{Chu} Y. Chushkin, F. Zontone, E. Lima, L. De Caro, P. Guardia, L. Manna, and C. Giannini.
 \emph{Three-dimensional coherent diffractive imaging on non-periodic specimens at the ESRF beamline ID10}. J. Synch. Rad. 21.3, pp. 594--599, 2014.


	\bibitem[CEH15]{Con} A. Conca, D. Edidin, M. Hering and C. Vinzant, \emph{An algebraic characterization of injectivity in phase retrieval}. Applied and Computational Harmonic Analysis, 38(2), pp. 346-356, 2015.
	
	\bibitem[Els03]{Els} V. Elser, \emph{Solution of the crystallographic phase problem by iterated projections}. Acta Crystallographica Section A: Foundations of Crystallography, 58(3), pp. 201-209, 2003.
	
	\bibitem[FHB01]{Faz} M. Fazel, H. Hindi and S. Boyd, \emph{A Rank Minimization Heuristic with Application to Minimum Order System Approximation}. Proceedings American Control Conference, 6:4734-4739, June 2001.
	
	\bibitem[Fie82]{Fie} J. R. Fienup, \emph{Phase retrieval algorithms: a comparison}. Appl. Opt. 21, pp. 2758-2769, 1982.
	
	\bibitem[FWd15]{Fog} F. Fogel, I. Waldspurger and A. d'Aspremont, \emph{Phase retrieval for imaging problems}. Mathematical Programming Computation, 8(3), pp. 311-335, 2015.
	
	\bibitem[GSa72]{Ger} R. W. Gerchberg and W. O. Saxton, \emph{A practical algorithm for the determination of the phase from image and diffraction plane pictures}. Optik, 35(2), pp. 237-246, 1972.
	
	\bibitem[GSB16]{Gol} T. Goldstein, C. Studer and R. Baraniuk, \emph{A Field Guide to Forward-Backward Splitting with a FASTA implementation}. Preprint, \href{https://arxiv.org/abs/1411.3406}{arXiv:1411.3406v6}
	
	\bibitem[Hay82]{Hay} M. Hayes, \emph{The reconstruction of a multidimensional sequence from the phase or magnitude of its Fourier transform}. IEEE Transactions on Acoustics, Speech, and Signal Processing, 30(2), pp. 140-154, 1982.
	
	\bibitem[Hof64]{Hof} E. Hofstetter, \emph{Construction of time-limited functions with specified autocorrelation functions}. IEEE Transactions on Information Theory, 10(2), pp. 119-126, 1964.
	
	\bibitem[HES16]{Hua} K. Huang, Y. C. Eldar and N. D. Sidiropoulos, \emph{On convexity and identifiability in 1-D Fourier phase retrieval}. 2016 IEEE International Conference on Acoustics, Speech and Signal Processing (ICASSP), 2016.
	
	\bibitem[Kec18]{Kec} M. Kech, \emph{Explicit frames for deterministic phase retrieval via PhaseLift}. Applied and Computational Harmonic Analysis, 45(2), pp. 282-298, 2018.
	
	\bibitem[Lew96]{Lew} A. S. Lewis, \emph{Convex Analysis on Hermitian Matrices}. SIAM J. Optim., 6(1), pp. 164-177, 1996.
	
	\bibitem[Luk05]{Luk} D. R. Luke, \emph{Relaxed averaged alternating reflections for diffraction imaging}. Inverse Problems, 21(1), pp. 37-50, 2005.
	
	\bibitem[Mar07]{Mar} S. Marchesini, \emph{A unified evaluation of iterative projection algorithms for phase retrieval}. Review of Scientific Instruments 78, 011301, 2007.
	
	\bibitem[MBK16]{Mare} S. Maretzke, M. Bartels, M. Krenkel, T. Salditt, and T. Hohage, \emph{Regularized Newton methods for x-ray phase contrast and general imaging problems}. Optics Express, 24(6), pp. 6490-6506, 2016.

	\bibitem[MSC98]{Mia} J. Miao, D. Sayre, HN Chapman, \emph{Phase retrieval from the magnitude of the Fourier transforms of nonperiodic objects}. JOSA A, 15(6), pp. 1662-1669, 1998.


	\bibitem[OCG19]{Ols} C. Olsson, M. Carlsson and D. Gerosa, \emph{Bias Reduction in Compressed Sensing}. Preprint, \href{https://arxiv.org/abs/1812.11329v1}{arXiv:1812.11329v1}.
	
	\bibitem[Osh12]{Osh} E. Osherovich. \emph{Numerical methods for phase retrieval}. Preprint, \href{https://arxiv.org/abs/1203.4756}{arXiv:1203.4756v1}.
	
	\bibitem[Pag06]{Pag} D. M. Paganin, \emph{Coherent X-Ray Optics}. Oxford Series on Synchrotron Radiation, 2006.

	\bibitem[RFP10]{Recht} B. Recht, M. Fazel and P. A. Parrilo, \emph{Guaranteed minimum-rank solutions of linear matrix equations via nuclear norm minimization.} SIAM review, 52(3), pp. 471-501, 2010.
	
	\bibitem[SCa91]{Sah} H. Sahinoglou and S. Cabrera, \emph{On phase retrieval of finite-length sequences using the initial time sample.} IEEE Transactions on Circuits and Systems, 38(5), pp. 954-958, 1991.
	
	\bibitem[SEC15]{She} Y. Shechtman, Y. C. Eldar, O. Cohen, H. N. Chapman, M. Jianwei and M. Segev, \emph{Phase Retrieval with Application to Optical Imaging: A contemporary overview}. Signal Processing Magazine, IEEE 32 (3), 87 (2015) 10.1109/MSP.2014.2352673.
	
	\bibitem[Sol14]{Sol} M. Soltanolkotabi, \emph{Algorithms and Theory for Clustering and Nonconvex Quadratic Programming}. PhD dissertation, 2014.
	
	
	


	
\end{thebibliography}
\end{document}